\documentclass[12pt,reqno]{amsart}
\usepackage{graphicx}
\usepackage{psfrag}
\usepackage{mathrsfs}
\usepackage{hyperref}
\usepackage{color}
\definecolor{darkgreen}{rgb}{0,0.51,0.11}
\usepackage{amsmath,amsthm,amsfonts,amssymb,amscd}

\numberwithin{equation}{section}

\def\ZZ{{\mathbb{Z}}}

\def\x{\tilde{x}}

\def\y{\tilde{y}}

\def\p{\bar{p}}
\def\fmx{\tilde{f}^{m}(\tilde{x})}
\def\fnx{\tilde{f}^{n}(\tilde{x})}
\def\fmy{\tilde{f}^{m}(\tilde{y})}
\def\fny{\tilde{f}^{m}(\tilde{y})}
\def\fnx*{\tilde{f}^{-n}(\tilde{x})}
\def\fmx*{\tilde{f}^{-m}(\tilde{x})}
\def\fmy*{\tilde{f}^{-m}(\tilde{y})}
\def\fny*{\tilde{f}^{-m}(\tilde{y})}

\def \ZZ {{\mathbb Z}}

\newtheorem{theorem}{Theorem}[section]
\newtheorem{corollary}[theorem]{Corollary}
\newtheorem{lemma}[theorem]{Lemma}
\newtheorem{proposition}[theorem]{Proposition}
\newtheorem{definition}[theorem]{Definition}
\newtheorem{remark}[theorem]{Remark}
\begin{document}

\thanks{ }

\author{P. Mehdipour}
\address{ICMC-USP S\~{a}o Carlos, Caixa Postal 668, 13560-970 S\~{a}o
  Carlos-SP, Brazil.} \email{Pouya@ufv.br}
%\urladdr{http://www.fing.edu.uy/$\sim$ures}

\author{A. Tahzibi}
\address{Departamento de Matem\'atica,
  ICMC-USP S\~{a}o Carlos, Caixa Postal 668, 13560-970 S\~{a}o
  Carlos-SP, Brazil.}
\email{tahzibi@icmc.sc.usp.br}\urladdr{http://www.icmc.sc.usp.br/$\sim$tahzibi}

\thanks{}



\renewcommand{\subjclassname}{\textup{2000} Mathematics Subject Classification}

%\date{\today}

\setcounter{tocdepth}{2}

\title[Surface Endomorphisms]{ SRB measures and Homoclinic relation for Endomorphisms}

\thanks{P.M was supported by Capes. A.T was partially supported by CNPq and  Fapesp. }
\begin{abstract}
In this paper we give an upper bound for the number of SRB measures of saddle type of local diffeomorphisms of boundaryless manifolds in terms of maximal cardinality of set of periodic points without any homoclinic relation. 
\end{abstract}

\maketitle
\tableofcontents
%------------------------------------------------------
\section{\textbf{Introduction}}
%----------------------------------------------------

The contrast between topological and measure theoretical properties is an interesting subject which frequently appears in the study of dynamics. 

In a beautiful simple construction, I. Kan \cite{5} gave an example of a local diffeomorphism $f$ defined on the cylinder $\mathbb{S}^1 \times [0,1]$ such that $f$ is topologically transitive and moreover, $f$ admits two SRB measures with intermingled basins. Besides the richness of intermingled basins property, the non-uniqueness of SRB measures joint with topological transitivity is amazing.

In \cite{13}, the authors proved that the above phenomenon can not exist for surface diffeomorphism. More precisely they proved 

\begin{theorem} \cite{13}
Let $f: M \rightarrow M$ be a $C^{1+\alpha}, \alpha > 0$ diffeomorphism over a compact surface $M.$ If $f$ is topologically transitive then there exists at most one (hyperbolic) SRB measure.
\end{theorem}

In this paper we deal with endomorphisms and by an endomorphism we refer to local diffeomorphism of a closed Riemannian manifold (compact and boundaryless). We remark that the endomorphism setting brings many surprises and distinctions with respect to the diffeomorphism context. Recall that I. Kan's example is made on two dimensional cylinder. His construction was extended by Ilyashenko, Kleptsy, Saltykov \cite{IKS}. See also \cite[11.1.1]{BDV}. 
By the way, it is possible to use Kan's example and construct a transitive \footnote{We would like to thank M. Andersson and J. Yang for pointing to us such construction of Kan example on $\mathbb{T}^2.$} endomorphism on $\mathbb{T}^2$ with two SRB measures of intermingled basins contrasting the above theorem for the case of non-invertible dynamics.

Here we find an upper bound for the number of ergodic SRB measures of {\it saddle type} in terms of the maximial cardinality of set of periodic points without any homoclinic relation. By a SRB measure of saddle type we mean a SRB measure whitout zero Lyapunov exponents and having both positive and negative ones.

For any endomorphism $f: M \rightarrow M$ we denote by $ \tilde{f}: M^f \rightarrow M^f$ the natural extension of $f$ where $\tilde{f}$ is  the shift map. The natural projection $\pi: M^f \rightarrow M$ is a semi conjugacy between $f$ and $\tilde{f}.$ Given any periodic point $p$ for an endomorphism $f$ we denote by $\bar{p}$ the unique point such that $\pi(\bar{p}) = p$ and $\bar{p}$ is periodic for $\tilde{f}.$ 

For any two hyperbolic periodic points $p$ and $q$ we say that $[p, q] \neq \emptyset$ iff $ W^u(\bar{p}) \pitchfork W_{loc}^s(\mathcal{O}(q)) \neq \emptyset.$  If $z \in W^u(\bar{p}) \pitchfork W_{loc}^s(\mathcal{O}(q))$ then $T_z W^u(\bar{p}) \oplus T_z(W_{loc}^s(\mathcal{O}(q))) = T_z(M).$ See Section \ref{stable-unstable} for more precise definitions. 

To give a concrete bound for the number of SRB measures we define {\it skeleton} inside hyperbolic periodic points of a fixed stable index.
\begin{definition}
A $k-$skeleton $( 0 < k < n=\dim(M))$ of $f$ is a subset of hyperbolic periodic points $\{p_i\}_{i \in \mathcal{I}}$ of stable index $k$ such that:
\begin{itemize}
\item For any hyperbolic periodic point $p \in M$ of index $k$, there is $i \in \mathcal{I}$ such that either $[p, p_i] \neq \emptyset$ or $[p_i, p] \neq \emptyset.$
\item For every $i \neq j, [p_i, p_j] = \emptyset.$
\end{itemize}
\end{definition}
Let us denote by $\mathcal{E}_k,$  the maximal cardinality of $k-$skeletons inside $Per_k(f)$ (hyperbolic periodic points of stable index $k$). 

\begin{theorem} \label{main}
Let $f: M \rightarrow M$ be a $C^2-$endomorphism of a closed $n-$dimensional manifold. Then for any $ 0 < k < n$
$$ \sharp \{ \text{Ergodic hyperbolic SRB measures of index k} \} \leq  \mathcal{E}_k.$$ 
\end{theorem}
 
The idea of using hyperbolic periodic points to analyze the number of SRB measures appear in \cite{13} and \cite{VY1}. In the context of partially hyperbolic diffeomorphisms with mostly contracting center, Viana and Yang \cite{VY} exhibited skeleton (defined by them with some similar properties) determining the number of basins of physical measures and concluded continuity results about the number of physical measures. In this paper we are not assuming any partial hyperbolicity assumption and invertibility of dynamics.  

Although the upper bound in the above theorem may be far from the number of SRB measures for a general endomorphism, in some cases we can obtain sharp number of SRB measures. For instance in the case of Kan example using the proof of the above theorem we conclude that there are at most two SRB measures (which is a known fact), See \ref{kan}. Indeed, in the proof of theorem \ref{main} we correspond to each ergodic hyperbolic SRB measure $\mu$ of index $k$ a hyperbolic periodic point $P_{\mu}$ in $Per_k$ whose {\it ergodic homoclinic class} has full measure. Then, the key point is that given any two hyperbolic SRB measures $\mu$ and $\nu$ of index $k,$ if for the corresponding periodic points $P_{\mu}, P_{\nu}$ one of the conditions: $[P_{\mu}, P_{\nu}] \neq \emptyset$ or $[P_{\nu}, P_{\mu}] \neq \emptyset$ is satisfied, then $\mu=\nu.$

We also mention a result of Hirayama-Sumi \cite{HS} where they prove the ergodicity of hyperbolic smooth (SRB) measures under the condition of constancy of the dimension of unstable bundle and intersection property of stable and unstable manifolds of almost every pair of regular points.
We emphasize that all the referred previous known results have been proved in the setting of invertible dynamical systems.

\section{Preliminaries on Endomorphisms} \label{preliminaries}

Let $M$ be a closed Riemannian surface. By a $C^2-$endomorphisms $f:M\rightarrow M$ we mean a local $C^{2}-$diffeomorphism and $\mathcal{M}_{f}(M)$ denotes the set of all $f-$invariant Borel probability measures. Note that $f$ satisfies the following integrability condition. 
$$\log \vert det\,d_{x}f\vert\in \mathcal{L}^{1}(M,\mu).$$  

For such $f$, consider the compact metric space
$$M^{f}:=\{\tilde{x}=(x_{n})\in\prod_{-\infty}^{\infty}M : f(x_{n})=x_{n+1} \quad \text{for all} \quad n\in \mathbb{Z}\},$$
equipped with the distance $\tilde{d},$ between $\tilde{x}=(x_{n})$ and $\tilde{y}=(y_{n})\in M^{f}$ defined by
$$\tilde{d}(\tilde{x},\tilde{y}):=\sum_{n=-\infty}^{\infty}2^{-|n|}d(x_{n}, y_{n}).$$
where $d$ is the distance on $M$ induced by the Riemannian metric. Let $\pi$ be the natural projection from $M^{f}$ to $M$ i.e, $\pi((x_{n}))=x_{0}, \forall \tilde{x}\in M^{f}$ and $\tilde{f}:M^{f}\rightarrow M^{f}$ be the shift homeomorphism. It is clear that  $\pi\circ \tilde{f}=f\circ \pi.$ 
%\textbf{$$\begin{CD}
% M^f @>\tilde f>> M^f\\
%@VV\pi V @VV\pi V\\
%M @>f>> M 
%\end{CD}$$}
The map $\tilde{f}: M^{f}\rightarrow M^{f}$ is called the \textit{Inverse Limit} of $f$ or the \textit{Natural Extension} of the system $(M,f)$ and $M^{f}$ is the \textit{Inverse Limit Space}. 

Any periodic point $p,$ i.e $f^n(p)=p$ has an special pre-image (under $\pi$) in $M^f,$ $\bar{p} = (\cdots, p, f(p), \cdots, f^{n-1}(p), \cdots)$ which is $\tilde{f}-$periodic. We work with this special pre-image in many instances.

The map $\pi$ induces a continuous map from $\mathcal{M}_{\tilde{f}}(M^{f})$ to $\mathcal{M}_{f}(M),$ usually denoted by $\pi_{*}$ i.e. for any $\tilde{f}-$invariant Borel probability measures $\tilde{\mu}$ on $M^{f}$, $\pi_*$ maps it to an $f-$invariant Borel probability measure $\pi_{*}\tilde{\mu}$ on $M$ defined as 
$$\pi_{*}\tilde{\mu}(\phi)=\tilde{\mu}(\phi\circ \pi),\,\,\,\,\,\,\forall \phi \in C(M).$$
The following proposition I.3.1 of \cite{10} guarantees that $\pi$ is a bijection between $\mathcal{M}_{\tilde{f}}(M^{f})$ and $\mathcal{M}_{f}(M).$ 
%\marginpar{put the reference, Ali}
\begin{proposition}\label{1}
	Let $f$ be a continuous map on $M$. For any $f-$invariant Borel probability measure $\mu$ on $M$, there exists a unique $\tilde{f}-$invariant Borel probability measure $\tilde{\mu}$ on $M^{f}$ such that $\pi_{*} \tilde{\mu}=\mu$. Moreover, $\mu$ is ergodic if and only if $\tilde{\mu}$ is ergodic.
\end{proposition}

\begin{proof}
	The above proposition is standard and we just recall the proof of correspondence between ergodic measures. 
	Consider the following diagram which permutes $\tilde f$ and $f$.
	\textbf{$$\begin{CD}
		M^f @>\tilde f>> M^f\\
		@VV\pi V @VV\pi V\\
		M @>f>> M 
		\end{CD}$$}
Suppose $\tilde{\mu}$ is ergodic. For each $f$-invariant subset $A\subset M$ i.e, $f^ {-1}(A)=A$, we can easily observe that $\pi^{-1}(A)$ satisfies $\tilde{f}^{-1}(\pi^{-1}( A)) =  \pi^{-1}( A)$ and by ergodicity of $\tilde{\mu}$ then $\tilde{\mu}(\pi^{-1}(A))=\pi_{*}\tilde{\mu}=\mu(A)$ is either zero or one. Now let prove the reciprocal claim. Consider $\mathcal{\tilde{B}}_n := \tilde{f}^n (\pi^{-1}\mathcal{B})$ where $\mathcal{\tilde{B}}$ is the Borel $\sigma-$algebra of $M.$ It is easy to see that $(M, \mathcal{B}, \mu, f)$ is isomorphic to $(\tilde{M}, \mathcal{\tilde{B}}_n, \tilde{\mu}, \tilde{f}).$ Observe that by a general statement for conditional expectations, for any $\tilde{\phi}\in L^{1}(\tilde{\mu})$ we have:
	$$
   E(\tilde{\phi} \circ \tilde{f}| \tilde{f}^{-1}(\mathcal{\tilde{B}}_n)) = E(\tilde{\phi}|\mathcal{\tilde{B}}_n) \circ \tilde{f}.	
	$$
 By invariance property of $\mathcal{\tilde{B}}_n, i.e, \tilde{f}^{-1}(\mathcal{\tilde{B}}_n) = \mathcal{\tilde{B}}_n$ we conclude that 
	$$
	 E(\tilde{\phi} \circ \tilde{f}| \mathcal{\tilde{B}}_n) = E(\tilde{\phi}| \mathcal{\tilde{B}}_n) \circ \tilde{f}.	
	$$
 Now, take any $\tilde{\phi} \in L^1(\tilde{\mu})$ which is $\tilde{f}-$invariant. By the above relation we have that $E(\tilde{\phi} | \mathcal{\tilde{B}}_n)$ is $\tilde{f}-$invariant. The $E(\tilde{\phi} | \mathcal{\tilde{B}}_n)$ can be considered as $\mathcal{B}$ measurable by isomorphism and ergodicity of $\mu$ implies that $E(\tilde{\phi} | \mathcal{\tilde{B}}_n)$ is constant. Finally $\mathcal{\tilde{B}}_n$ converge to the Borel $\sigma-$algebra of $\tilde{M}$ and this implies that $\tilde{\phi} = \lim E(\tilde{\phi} |\mathcal{\tilde{B}}_n )$ is an almost everywhere constant function.

	%let by contradiction suppose that $\mu$ is ergodic but $\tilde{\mu}$ is not ergodic. $M^{f}$ is a compact metric space and $\tilde{\mu}\in \mathcal{M}_{\tilde{f}}(M^{f})$. By ergodic decomposition theorem:
	%\[\tilde{\mu}=\int_{M^{f}}\tilde{\mu}_{\x}\,d\tilde{\mu}(\tilde{x})\]
%	where $\tilde{\mu}_{x}$ is $\tilde{f}-$invariant and ergodic for $\tilde{\mu}-$a.e. $\tilde{x}\in M^{f}$. 
	%By the first part of proposition, $\tilde\mu$ is unique $\tilde{f}-$invariant measure that $\pi_{*}\tilde{\mu}=\mu$. It implies that 
	%\[\mu=\pi_{*}(\tilde{\mu})=\int_{M^{f}}\pi_{*}\tilde{\mu}_{x}\,d\tilde{\mu}(\tilde{x})\]
	%The projection of almost all ergodic component $\tilde{\mu}_{\tilde{x}}$ either will be $\mu$ or we get some non-trivial ergodic components  for $\tilde{\mu}.$ This gives the desires result.

\end{proof}

\subsection{Multiplicative Ergodic Theorem on Natural Extension}

Let $\mu$ be an $f-$invariant
Borel probability measure on $M$. We denote by $\tilde{\mu}$ the $\tilde f-$invariant Borel probability measure
on $M^{f}$ such that $\pi_{*}\tilde{\mu}=\mu$. There exists a full measure subset $\tilde{\mathcal{R}}$ called set of {\it regular points} such that for all
$\tilde{x}=(x_{n})\in \tilde{\mathcal{R}}$ and $n\in \mathbb{Z}$ the tangent space $T_{x_{n}}M$ splits into a direct sum
$$T_{x_{n}}M=E_{1}(\tilde{x}, n)\oplus\cdots\oplus E_{r(x_{0})}(\tilde{x}, n)$$
and there exists 
$-\infty<\lambda_{1}(\tilde x)<\cdots<\lambda_{r(\tilde x)}<\infty$ and $m_{i}(\tilde x)$ ($i=0,1,...,r(\tilde{x})$) such that:

\begin{enumerate}
	\item dim $E_{i}(\tilde{x}, n)=m_{i}(\tilde x)$;
	
	\item $D_{x_{n}}f(E_{i}(\tilde{x}, n))= E_{i}(\tilde{x}, n+1)$, and  $D_{x_{n}}f|_{E_{i}(\tilde{x},n)}:E_{i}(\tilde{x}, n)\rightarrow E_{i}(\tilde{x}, n+1)$ is an isomorphism. For $v\in E_{i}(\tilde{x}, n)\backslash \{0\}$,
	
	\begin{center}
		$\begin{cases}
		\lim_{m\rightarrow\infty}\frac{1}{m}\log\Vert D_{x_{n}}f^{m}(v)\Vert=\lambda_{i}(\tilde{x});\\
		
		\lim_{m\rightarrow\infty}-\frac{1}{m}\log\Vert(D_{x_{n- m}}f^{m}|_{E_{i}(\tilde{x},n-m)})^{-1}(v)\Vert=\lambda_{i}(\tilde{x});
		\end{cases}$
	\end{center} 
	\item if $i\neq j$ then
	$$\lim_{n\rightarrow\pm\infty} \frac{1}{n}\log\sin\angle(E_{i}(\tilde{x}, n), E_{j}(\tilde{x}, n))=0,$$
	where $\angle(V, W)$ denotes the angle between sub-spaces $V$ and $W$.
	\item $r(.), \,\lambda_{i}(.)\,\,and\,\,m_{i}(.)$ are measurable and $\tilde f-$invariant. Moreover $r(\tilde x)=r(x_{0}),\,\lambda_{i}(\tilde x)=\lambda_{i}(x_{0})$ and $m_{i}(\tilde{x})=m_{i}(x_{0})$ for all $i=1,2,...,r(\tilde x)$.
\end{enumerate}

From now on we work with ergodic measures and the Lyapunov exponents are constant almost everywhere with respect to the reference measure. The celebrated Pesin's blocks are defined naturally in the non-invertible case in the limit inverse space. We use a simple definition which is enough for our purpose. Let $\mu$ be an ergodic invariant measure and $\lambda$ (resp. $\theta$) the least in modulus positive (resp. negative) Lyapunov exponent. Suppose that $\mu$ has $k$ negative Lyapunov exponents.

\begin{definition}[\textbf{Pesin Blocks}]\label{1-08}
	Fix $0 < \epsilon \ll 1.$ For any $l>1$, we define a Pesin block $\tilde{\Delta}_{l}$ of $M^{f}$ consisting of
	$\tilde{x}=(x_{n})\in M^{f}$ for which there exists a sequence of splittings $T_{x_{n}}M=E^{s}(\tilde{x}, n)\oplus E^{u}(\tilde{x}, n)$,
	$n\in \ZZ$, satisfying:
	
	\begin{itemize}
		\item $\dim E^{s}(\tilde{x}, n)=k$ ;
		\item $D_{x_{n}}f(E^{s}(\tilde{x},n))= E^{s}(\tilde{x}, n+1),$ $D_{x_{n}}f(E^{u}(\tilde{x}, n))=E^{u}(\tilde{x}, n+1)$;
		\item for $m\geq 0, \, v\in E^{s}(\tilde{x}, n)$ and $w\in E^{u}(\tilde{x}, n);$\\
		
		$\begin{cases}
		\Vert D_{x_{n}}f^{m}(v)\Vert\leq e^{l} e^{-(\theta-\epsilon)m } e^{(\epsilon |n|)}\Vert v \Vert,\forall n\in\ZZ,n\geq 1\\
		
		\Vert (D_{x_{n-m}}f^{m}|_{E^{u}(\tilde{x},n-m)})^{-1}(w) \Vert\leq e^{l}e^{-(\lambda-\epsilon ) m} e^{(\epsilon |n-m|)}\Vert w \Vert,\forall n\in\ZZ,n\geq 1;
		\end{cases}$\\
		
		\item sin $\angle(E^{s}(\tilde{x}, n),$ $E^{u}(\tilde{x}, n))\geq e^{-l}e^{-\epsilon |n|}$.
	\end{itemize}
 In the above definition it is enough to take $\epsilon$ less than $\frac{1}{2} \min \{ \lambda, \theta \}.$
\end{definition}
Pesin blocks are compact subsets of $M^{f}$ where the subspaces $E^{s}(\tilde{x}, n)$ and $E^{u}(\tilde{x}, n)$ of $T_{x_{n}}M$ depend continuously on $\tilde{x}$ and $\tilde{f}^{\pm}(\tilde{\Delta}_{l})\subset\tilde{\Delta}_{l+1}$.

\section{Stable, Unstable Sets, SRB Property} \label{stable-unstable}
After the works of Pesin on general theory of stable and unstable manifolds for non-uniformly hyperbolic diffeomorphisms 
(see \cite{7}), P.-D Liu and M. Qian \cite{36} developed a rigorous related theory for random diffeomorphisms. Using similar techniques, Sh. Zhu proved an unstable manifold theorem for non-invertible differentiable maps of finite dimension \cite{21} (see \cite{10} for more details.)  Here we would like to emphasize the differences between unstable and stable sets (and manifolds).

\begin{definition}[\textbf{Local Unstable Manifold}]\label{1-16}
	
 Let $\tilde x\in\tilde{\mathcal{R}}$ and $\lambda$ the least positive Lyapunov exponent of $\mu$. We call $W_{loc}^{u}(\tilde x)$ a \textit{local unstable manifold} of $f$ at $\tilde x$ when exists a $u-$dimensional $C^2-$embedded sub-manifold of $M$ ($u$ is the number of positive Lyapunov exponents.) such that there are $\epsilon, C >0$, and for any $y_{0}\in W^{u}_{loc}(\x)$, there exists a unique $\tilde{y}=\{y_{n}\}_{n\in \mathbb{Z}}\in M^{f}$ such that $\pi(\tilde y)=y_0$ and $\forall n\in \mathbb{N},$
	$$d(y_{-n},x_{-n})\leq C\,e^{-n(\lambda-\epsilon)}\,d(x_0,y_0)$$
 Moreover we define the \textbf{local unstable set} of $\tilde{f}$ at $\tilde x=(x_n)$ as  
	$$\widetilde{W}^{u}_{loc}(\tilde x):=\{\tilde y \in M^{f}: y_0\in W^{u}_{loc}(\tilde x),d(y_{-n},x_{-n})\leq C\,e^{-n(\lambda-\epsilon)}\,d(x_0,y_0)\}.$$
\end{definition}
It comes out that $\pi(\widetilde{W^{u}_{loc}}(\tilde{x}))=W^{u}_{loc}(\tilde{x})$ is the local unstable manifold of $\x.$ 

\begin{definition}[\textbf{Unstable Manifold}]\label{1-17}
 The unstable manifold of $f$ corresponding to $\tilde x\in \tilde{\mathcal{R}}$ is defined as
	
	\begin{equation}\label{1-10}
		W^{u}(\tilde x)=\{y_0\in M|\,\exists\tilde y\in M^{f}\,\,with\,\,\pi\tilde y=y_0,\,\,and\,\,\overline\lim_{n\rightarrow+\infty}\,\frac{1}{n}log\,d(x_{-n},y_{-n})<0\}\}
	\end{equation}
	
 and we will write 
	
	\begin{equation}\label{1-11}
		\widetilde{W}^{u}(\tilde x)=\{\tilde y\in M^{f}|\overline\lim_{n\rightarrow+\infty}\,\frac{1}{n}log\,d(x_{-n},y_{-n})<0\}.
	\end{equation}
	
Notice that $\pi(\widetilde{W}^{u}(\tilde x))= W^{u}(\tilde x)$ and the global unstable manifold $W^{u}(\tilde x)$ is the union of forward images of local unstable manifolds at $x_{-n}.$
	%Notice that the unstable manifold corresponded to a trajectory $\tilde{x}\in\tilde{\mathcal{R}}$ can be defined as the union of an increasing sequence of images of $C^{1}$ embedded disks:
	%
	%\begin{equation}
	%W^{u}(\tilde{x})=\bigcup_{n=0}^{\infty}f^{n}(W^{u}_{loc}(\tilde{f}^{-n}(\tilde{x}))).
	%\end{equation}
	
\end{definition}

%\begin{definition}[\textbf{Global Unstable Set}]\label{1-18}
% We define the global unstable set of $f$ at a point $x$, as
	
%	\begin{equation}\label{1-12}
%		W^{u}(x):=\bigcup_{\tilde x \in \pi^{-1}(x)\cap \tilde{\mathcal{R}}}\,W^{u}(\tilde x).
%	\end{equation}
	
%\end{definition}

\begin{figure}[]\label{Fig:US}
\minipage{0.50\textwidth}
  \includegraphics[width=0.9\linewidth]{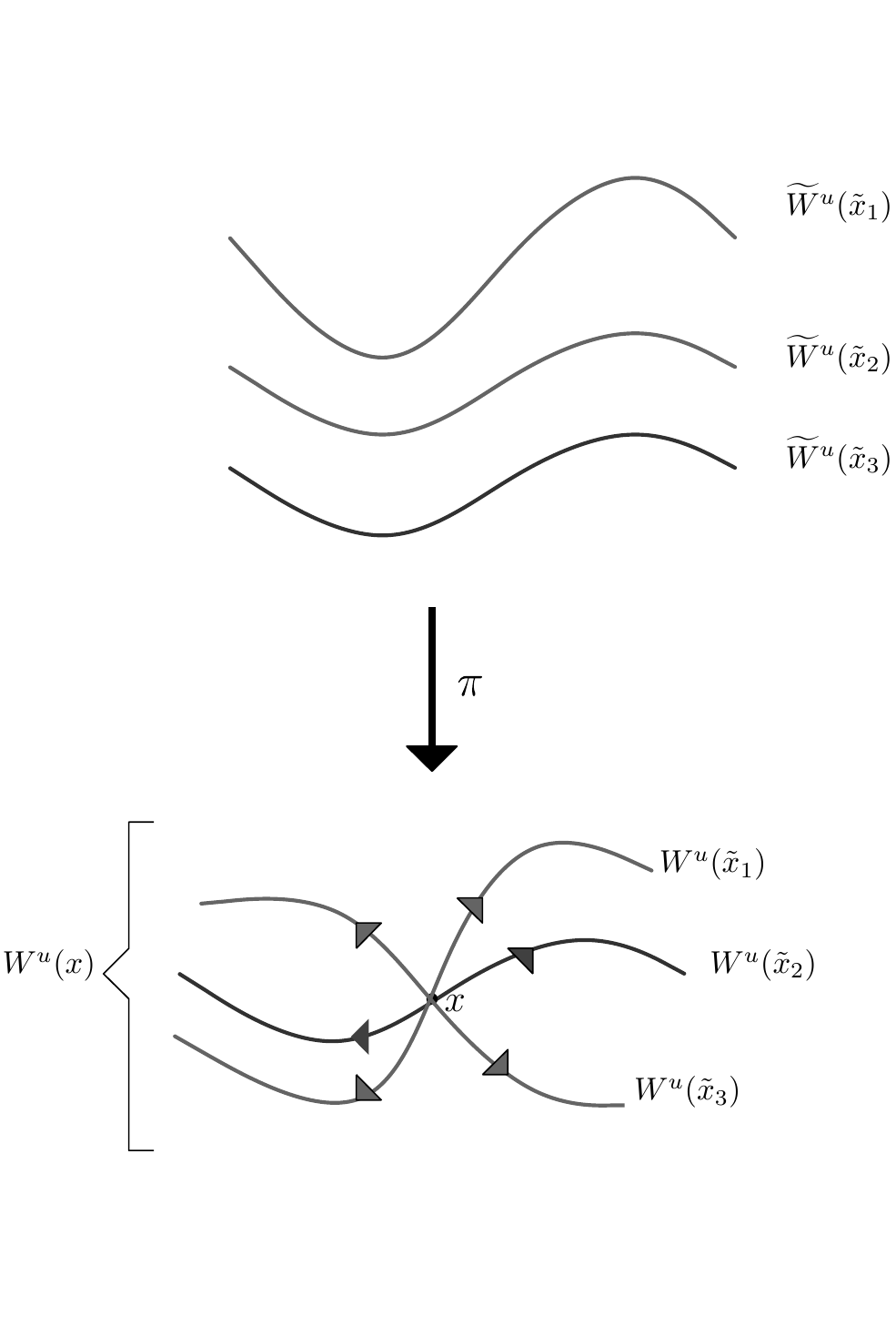}
		
\endminipage\hfill
\minipage{0.50\textwidth}
  \includegraphics[width=1\linewidth]{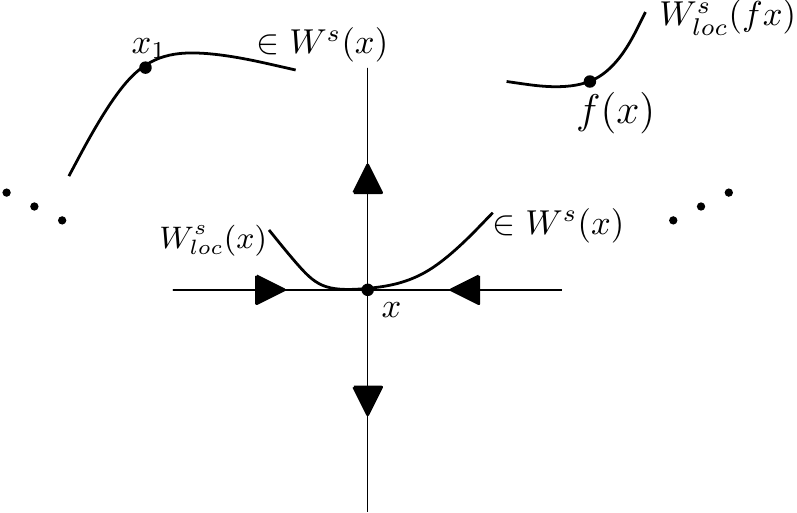}
  %\caption{A really Awesome Image}\label{fig:U}
\endminipage\hfill
\caption{unstable manifolds of different trajectories and stable set of a point ($x,x_1\in f^{-1}(fx)$).}
\end{figure}

%\begin{figure}[h]\label{US}
%	\centering
%	\begin{subfigure}
%		\centering
%		\includegraphics[scale=0.7]{S01.pdf}
		%\caption[b]{Unstable set and unstable manifolds of different trajectories .}
%		\label{fig:S.pdf}
%	\end{subfigure}
%	\begin{subfigure}
%		\centering
%		\includegraphics[scale=0.5]{U01.pdf}
%		%\caption[b]{Unstable set and unstable manifolds of different trajectories .}
%	 	\end{subfigure}
%	\caption{stable set of a point and unstable manifolds of different trajectories }
%\end{figure}
Observe that the global unstable set is not necessarily a manifold and it is defined for $x \in M$ and not $\x \in M^f.$

\textbf{Local Stable Manifolds:}
 Besides the case of unstable manifolds, the local stable manifolds are defined uniquely for $x \in M$, when they exist. In fact $ W^s_{loc}(\x)$  is defined exactly as in the case of invertible dynamics. So, we may use the notation $W^s_{loc}(x)$ or $W^s_{loc}(\x)$ for the same object. In this text, we use $W^s_{loc}(\x).$
However, the global stable {\bf set} is defined as follow:
$$
W^{s}(x)=\bigcup^{+\infty}_{n=0}f^{-n}(W^s_{loc}({f^n(x)}))
$$
Observe that the global stable set is not necessarily a connected {\bf manifold} (see figure 1).

Using definitions and local stable-unstable manifolds theorems from \cite{10}, one concludes the following invariance properties :
\begin{itemize}
	\item $f(W^{s}(\x))=W^{s}(\tilde{f}(\x))$;
	\item $f(W^{u}(\x))=W^{u}(\tilde{f}(\x)).$
\end{itemize}

\subsection{Heteroclinic Relation and Incliniation Lemma}
\label{homoclinici}
Let $p, q$ be two hyperbolic periodic points. We write $[p, q] \neq \emptyset$ iff $W^u(\bar{p}) \pitchfork W^s_{loc}(\mathcal{O}(q)) \neq \emptyset.$

In what follows we recall the well known inclination lemma. In typical texts in dynamics, the inclination lemma is proved for invertible dynamics. However, there is a small subtle difference between the invertible and non-invertible case. 

\begin{lemma}[\textbf{$\lambda-$Lemma}]\label{4-7}
Let $p$ be a hyperbolic fixed point of map $f$ and $D \subset W^u (\bar p)$ be a compact disk. If $\Sigma$ is an embedded $C^1$ sub-manifold of $M$ intersecting $W^s_{loc}(p)$ transversally, then for any large $n,$  $f^n(\Sigma)$ contains an embedded manifold $\Sigma_n$ which is $C^1-$close to $D$  where $\bar p=(...pppp...)\in M^f$. %Similarly if $\acute{\Sigma}$ is an embedded $C^1$ submanifold of M intersecting $W^u_{\delta}(\tilde (p))$ transversely near p, then  $f^{-n}(\acute{\Sigma})$ contains an embedded manifold $\acute{\Sigma_n}$ which is $C^1$ close to $W^s_{\delta}(p)$ for n large enough.
\end{lemma}

The proof is similar to the diffeomorphisms case. We just observe that $E^{u}(\p):=\bigcap_{k=0}^{\infty}Df_{p}^{k}(C^u(p))$ where $C^u(p)$ is the complement of a thin cone around $E^s(p).$ As for any point in a small neighborhood of $\Sigma \cap W^s_{loc}(p)$ the tangent space is outside the stable cone, the same argument of usual $\lambda-$lemma yields the proof. 
\subsection{SRB property}

In this subsection we review the definitions and basic properties of SRB measures.

\begin{definition}[\textbf{Unstable Partition Sub-Ordinate to $W^{u}$}]\label{def: unstable partition}
 A measurable partition $\zeta$ of $M^f$ is said to be \textit{subordinate to $W^u$ manifolds of $(f,\mu)$} if for $\tilde\mu$-a.e. $\tilde x$, $\zeta(\tilde x)$ has the following properties:
	\begin{itemize}
		\item $\pi_|\zeta(\tilde x): \zeta(\tilde x)\rightarrow \pi(\zeta(\tilde x))$ is bijective
		\item There exists a $k(\tilde x)$ dimensional $C^2-$embedded sub-manifold $W_{\tilde x}$ of $M$, such that $W_{\tilde x}\subset W^u(\tilde x),\,\,\pi(\zeta(\tilde x))\subset W_{\tilde x}$ and $\pi(\zeta(\tilde x))$ contains an open neighborhood of $x$ in $W_{\tilde x}$, this neighborhood being taken in the topology of $W_{\tilde x}$ as a sub-manifold of $M$.
	\end{itemize}
\end{definition}
It is always possible to construct measurable partitions sub-ordinated to unstable manifolds (see \cite{10}, \cite{6}.) 

\begin{definition}[\textbf{SRB Property}]\label{def: SRB property}
 An $f-$invariant measure $\mu$ has the SRB property, if for every measurable partition $\zeta$ of ${M^f}$ sub-ordinate to ${W^u}-$manifolds of $(f,\mu)$ we have, for $\tilde \mu$-a.e. $\tilde x \in M^f$,
	$$
	\pi(\tilde {\mu}_{\tilde x}^{\zeta})\prec\prec m_{\tilde x}^{u},
	$$
where $\{\tilde {\mu}_{\tilde x}^{\zeta}\}_{\tilde x\in M^f}$ is a canonical system of conditional measures of $\tilde\mu$ associated with $\zeta$, $\pi(\tilde {\mu}_{\tilde x}^{\zeta})$ is the projection of $\tilde {\mu}_{\tilde x}^{\zeta}$ under $\pi_|\zeta(\tilde x): \zeta(\tilde x)\rightarrow \pi(\zeta(\tilde x))$ and $m_{\tilde x}^{u}$ denotes the Lebesgue measure on $W_{loc}^{u}(\tilde x)$ induced by its inherited Riemannian structure.
\end{definition}

In the diffeomorphisms context, F. Ledrappier and L.-S. Young \cite{4} proved that a measure satisfies Pesin entropy formula, if and only if it is absolutely continuous with respect to (Pesin) unstable manifolds. Moreover, they showed that the densities $d\mu^{u}_{x}/dm^{u}_{x}$ are given by \textbf{strictly positive} functions that are $C^{1}$ along unstable manifolds. ($\mu^{u}_{x}$ is the unstable conditional measure and $m^{u}_{x}$ the induced Lebesgue inherited from Riemannian structure.)

The same results (adapted to the non-invertible case) hold for endomorphisms. This implies that $\pi(\tilde {\mu}_{\tilde x}^{\zeta})\prec\prec m_{\tilde x}^{u}$ in the above definition can be substituted by $\pi(\tilde {\mu}_{\tilde x}^{\zeta}) \approx m_{\tilde x}^{u}$ for $\tilde{\mu}$ almost every $\tilde{x}.$

\section{Admissible Manifolds and Katok Closing Lemma}\label{4}
In this section we try to give simplified definition of admissible manifolds which are tools in the proof of the main theorem. To define admissible manifolds we recall the definition of Lyapunov charts in the non-invertible case of Pesin theory.

Let $f$ be a $C^{2}-$endomorphism of a closed Riemannian surface $M$. Assume that we have a non-empty Pesin block $\tilde{\Delta}_{l}$ for $l>1$ for an ergodic measure with $k-$negative Lyapunov exponents and $(n-k)-$positive Lyapunov exponents. We can change the metric on $\tilde{\Delta}_{l}$ so that $f|_{\tilde{\Delta}_{l}}$``looks uniformly hyperbolic". This happens by replacing the induced Riemannian metric on $T_{x_{n}}M$ ($\tilde{x}=(x_{n})\in \tilde{\Delta}_{l}$) by a new metric which is called Lyapunov metric. 

Let $0<\lambda^ {'}<\mu^{'}<\infty$ such that
\begin{equation}\label{2-00}
\lambda^{'}=\lambda-2\epsilon, \quad \mu^{'}=\mu-\, 2\epsilon.
\end{equation}
%Define the new inner product $<.,.>_{\tilde{x}}$ depending to some $\tilde{x}\in \tilde{\RR}$ as follows
%\begin{equation}
%<v,w>_{\tilde{x}}=\sum_{k=0}^{\infty}<df^{k}v,df^{k}w>_{\tilde{f}^{k}\tilde{x}}\acute{\lambda}^{-2k}
%\end{equation}
%if $v,w\in E^{s}(\tilde{x})$, and
%\begin{equation}
%<v,w>_{\tilde{x}}=\sum_{k=0}^{\infty}<df^{-k}v,df^{-k}w>_{\tilde{f}^{-k}(\tilde{x})}\acute{\mu}^{2k}
%\end{equation}
%if $v,w\in E^{s}(\tilde{x})$. (Note that in above equation $f^{-k}$ is the inverse function in the sense of $\tilde{x}$.) By Cauchy-Schwartz inequality and using \ref{8} it is easy to see that both series converges. In fact for the first
We can define the new metric $<.,.>_{\tilde{x}}^{\prime}$ as:
\begin{equation}\label{2-1-00}
< v_{s},w_{s}>_{\tilde{x}}^{\prime}:=\sum_{m=0}^{\infty}< df^{m}(v_{s}),df^{m}(w_{s})>_{x_n}e^{2\lambda^{\prime}m};
\end{equation}
where $v_s,w_s\in E^{s}(\tilde{x},0)$, and
\begin{equation}\label{2-000}
< v_{u},w_{u}>_{\tilde{x}}^{\prime}:=\sum_{m=0}^{\infty}< df^{-m}(v_{u}),df^{-m}(w_{u})>_{x_{-n}}e^{2\lambda^{\prime}m};
\end{equation}
where $v_u,w_u\in E^{u}(\tilde{x},0)$.
%Notice that here $f^{-1}_{\tilde{x}}(x)$ means the pre-images along the orbit $\tilde{x}$.
Now for $(v,w)\in T_{x_{0}}M$ that $v=v_s+v_u,w=w_s+w_u$, define
\begin{equation}\label{Lyp metric}
<v,w>_{\tilde{x}}^{\prime}:= \max\{< v_{s},w_{s}>_{\tilde{x}}^{\prime},< v_{u},w_{u}>_{\tilde{x}}^{\prime}\}.
\end{equation}
The new metric induces a new norm  $\Vert . \Vert^{\prime}_{\tilde x}$ on $T_{x}M$:
%For $v=v_{s}+v_{u}\in T_{x_{0}}M$ with $v_{s}\in E^{s}(\tilde x,0),\,v_{u}\in E^{u}(\tilde x,0)$

\begin{align*}
&\Vert v_{s}\Vert_{\tilde{x}}^{\prime}=(\sum_{m=0}^{\infty}e^{2\lambda^{\prime}m}\Vert d_{x}f^{m}(v_{s})\Vert^{2})^{1/2}
,\\
&\Vert v_{u}\Vert_{\tilde{x}}^{\prime}=(\sum_{m=0}^{\infty}e^{\mu^{\prime}m}\Vert(d_{x}f^{m}|_{E^{u}(\tilde{x},0)})^{-1}(v_{u})\Vert^{2})^{1/2}
,\\
&\Vert v\Vert_{\tilde{x}}^{\prime}=\max\{\Vert v_{s}\Vert^{\prime}_{\tilde x},\, \Vert v_{u}\Vert_{\tilde{x}}^{\prime}\}.
\end{align*}

%\begin{align}\label{2-1-0}
%{\Vert v_{s}\Vert_{\tilde{x}}^{\prime}}^{2}&=\sum_{m=0}^{\infty}e^{2(\lambda-2\epsilon)m}
%\Vert d_{x_{0}}f^{m}(v_{s})\Vert^{2}\notag\\
%&\leq \sum_{m=0}^{\infty}e^{2(\lambda-2\epsilon)m}e^{2\epsilon l}e^{-2(\lambda-\epsilon)m}\Vert v_{s}\Vert^{2}\notag \\
%&=\Vert v_{s}\Vert^{2}e^{2\epsilon l}(\sum_{m=0}^{\infty}e^{-\epsilon m})^{2}\\
%&\Rightarrow \Vert v_{s}\Vert_{\tilde{x}}^{\prime}\leq \Vert v_{s}\Vert e^{\frac{l\chi}{100}}\sum_{m=0}^{\infty}e^{\frac{-\chi}{100}m}< \infty.\notag
%\end{align}

%Similarly for $\Vert v_{u}\Vert_{\tilde{x}}^{\prime}$.
%As it was mentioned in above lines, the new metric exhibits a local "hyperbolicity" which is independent of $\tilde{x}\in \tilde{\Delta}_{l}$. For example 

One can verify that for $v_{s}\in E^{s}(\tilde x,0),\,v_{u}\in E^{u}(\tilde x,0)$:
$$
\Vert d_{x_{0}}f(v_{s})\Vert^{\prime}_{\tilde{f}(\tilde x)}\leq e^{-\lambda^{'}}\Vert v_{s} \Vert^{\prime}_{\tilde x}
,$$
\begin{equation*}
\Vert d_{x_{0}}f(v_{u})\Vert^{\prime}_{\tilde{f}(\tilde x)}\geq e^{\mu^{'}}\Vert v_{u} \Vert^{\prime}_{\tilde x}.
\end{equation*}

There exist also the following estimate on the norms ($\Vert.\Vert$ is the induced Riemannian norm on $T_{x}M$).
\begin{equation}\label{2-1}
\frac{1}{2}\Vert v \Vert\leq\Vert v \Vert_{\tilde{x}}^{\prime}\leq \textit{a}_{l} \Vert v \Vert\,\,\,\,\,\,\,\,\,\,\forall\tilde{x}=(x_{n})\in\tilde{\Delta}_{l}
\end{equation}

Usually $\Vert . \Vert^{\prime}_{\tilde{x}}$ is called \textbf{Lyapunov norm}.
%\begin{remark}
%Although defining such new metric involves some loss of accuracy on comparision of two metrics on $\tilde{\Delta}_{k}$ as k increases but once we have defined this new metric on a non-empty Pesin set, then $\tilde{f}:\tilde{\Delta}_{k}\rightarrow\tilde{\Delta}_{k}$ becomes a uniformly hyperbolic and it removes the need for constants $C=l$ in the definition of blocks $\tilde{\Delta}_{k}$ which is a good advantage in local form.
%\end{remark}
The following proposition is about the existence of Lyapunov charts in the context of local diffeomorphisms and its proof is a simple adaptation of proposition (2.3) of \cite{8}. %(Figure \ref{fig:chart})
\begin{proposition}\label{2-01}
There exists a number $r>0$ so that for every point $\tilde{x}\in \tilde{\Delta_{l}}$ we can find a neighborhood $B(\tilde{x})$ around the point $x=\pi(\tilde{x})$ and a diffeomorphism $\Phi_{\tilde{x}}:B^k_{r}\times B^{\dim(M)-k}_{r}\rightarrow B(\tilde{x})$($B^{d}_{r}$ is Euclidean closed disc of radius $r$ around the origin in $\mathbb{R}^d$).
Also there exists a family of $C^{1}-$maps $F_{\tilde{x}}: B^k_{r}\times B^{\dim(M)-k}_{r}\rightarrow \mathbb{R}^k\times \mathbb{R}^{\dim(M)-k}$ satisfying the following properties:
\begin{enumerate}
  \item $\Phi_{\tilde{x}}(0)=\pi(\tilde{x});\,\,\,\,\, $
  \item $F_{\tilde{x}}(z)=\Phi_{\tilde{f}(\tilde{x})}^{-1}\circ f\circ\Phi_{\tilde{x}}(z)\ \ \ \ \ \ \ \ \ \ \ \ \ \ \ \ $ \text{for} $z=(u,v);$
  \item $F_{\tilde{x}}$ has the form:
$$F_{\tilde{x}}(u,v)=(A_{\tilde{x}}\,u+h^{1}_{\tilde{x}}(u,v),\,B_{\tilde{x}}\,v+h^{2}_{\tilde{x}}(u,v)),$$
such that:
$$h^{2}_{\tilde{x}}(0,0)=h^{2}_{\tilde{x}}(0,0)=0,\,\,\,dh^{1}_{\tilde{x}}(0,0)=dh^{2}_{\tilde{x}}(0,0)=0$$ \\
and
$$\Vert A_{\tilde{x}}\Vert\leq\,e^{- \lambda^{'}}, \,\,\,\,\,\,\,\,\,\,\Vert B_{\tilde{x}}\Vert\geq\, e^{\mu^{'}}.$$
(all the norms are considered as Euclidean.)

%Lets take $\lambda(\chi)=\max \{ 1/2,\, e^{\frac{-99}{100}\chi}\}$. Then
For $z\in B^k_{r}\times B^{\dim(M)-k}_{r}$ let $ h_{\tilde{x}}(z)=(h^{1}_{\tilde{x}}(z),h^{2}_{\tilde{x}}(z))$, then
$$
\Vert (dh_{\tilde{x}})_{z_{1}}-(dh_{\tilde{x}})_{z_{2}} \Vert\leq \Upsilon\,\textit{a}_{l}\, \Vert z_{1}-z_{2}\Vert$$ where $\Upsilon$ is an absolute constant.

  \item the metric $\Vert . \Vert ^{\prime}_{\tilde{x}}$ depends continuously on $\tilde{x}$ over any Pesin block $\tilde{\Delta}_{l}$.
%  \item The decomposition $T_{z}M=d\Phi_{\x}(\mathbb{R})\times d\Phi_{\x} (\mathbb{R})$ depends continuously on $\x$ for such $\x\in \tilde{\Delta}_{l}$ that $z\in B(\x)$.

\end{enumerate}
\end{proposition}

%\begin{figure}
%\centering
%\includegraphics[width=0.45\linewidth]{chart.pdf}
%\caption{}
%\label{fig:chart}
%\end{figure}

%We may diminish the size of the neighborhood $B(\tilde{x})$ and change it to $R(\tilde{x})=\Phi_{\tilde{x}}(B_{\eta_{l}}\times B_{\eta_{l}})$. Let suppose
%\begin{equation}\label{lambda}
%\lambda(\chi)=\max \{ 1/2,\, e^{\frac{-99}{100}\chi},\}
%\end{equation}
%then for $\x\in \tilde{\Delta}_{l},$
%\begin{equation*}
%\eta_{l}= \dfrac{(1-\lambda(\chi))^{2}}{100}(2\,\Upsilon)^{-1}(\textit{a}_{l}^{-1}).
%\end{equation*}

%by choosing $\chi=\min_{i}\vert\lambda_{i}(x)\vert$ as in \ref{1-008}, also by ergodicity of $\mu\in \mathcal{M}_{erg}^{*}$ for $\tilde{x},\tilde{f}^{m}(\tilde{x})\in\tilde{\Delta}_{l}, m\in \ZZ^{+}$ we have $\eta(\tilde{f}^{m}(\tilde{x}))=>0$.)
%This new $R(\tilde{x})$ is called \textbf{standard $\tilde{x}-$box} or \textbf{Lyapunov chart}.

\begin{figure}
\centering
\includegraphics[width=0.45\linewidth]{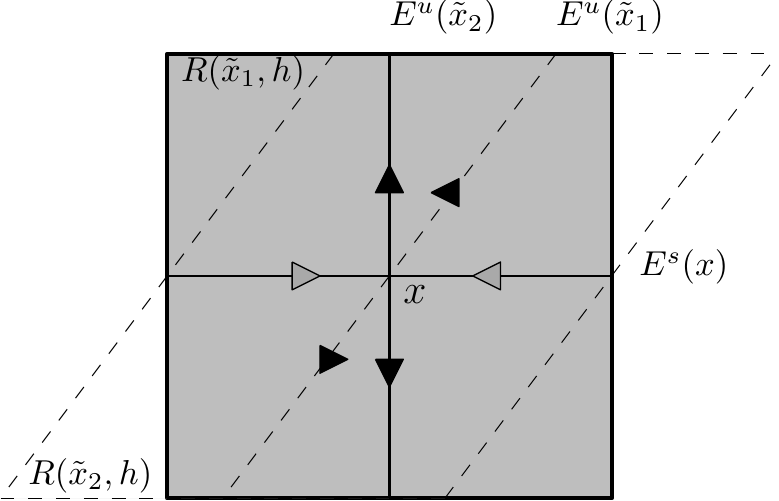}
\caption{unstable manifolds for  $\tilde{x}_1\neq \tilde{x}_2$ such that $\pi(\tilde{x}_1)=\pi(\tilde{x}_2)=x$.}
\label{fig:Lyp chart}
\end{figure}

%\begin{remark}
% Suppose $\tilde{x}_1\neq\tilde{x}_2\in \tilde{\RR}$ and $\pi^{-1}(\tilde{x}_i)=x,(i=1,2)$. Figure \ref{fig:Lyp chart} is showing standard $\tilde{x}_1$-box and standard $\tilde{x}_{2}-$ box
%that by definition \ref{1-08}, they are $df-$invariant along their own trajectory.
%\end{remark}

The set of  admissible manifolds is just the set of graph of $C^1-$functions defined locally and with bounded Lipschitz constant. The difference with the invertible setting is that they depend on $\x$ and not just $\pi(\x) = x.$

To be more precise, we define the class of  $(\gamma, \delta, h)$ stable-admissible and $(\gamma, \delta, h)$ unstable-admissible manifolds  close to $x$ as follow: 
\begin{equation}
S_{\tilde{x}}^{\gamma,\delta,h}=\{\Phi_{\tilde{x}}(graph\, \phi) | \phi\in C^{1}(B^{k}_{h},\,B^{n-k}_{h}),\Vert\phi(0)\Vert\leq\delta,\,\Vert d\,\phi\Vert\leq\gamma\}
\end{equation}
\begin{equation}\label{2-002}
U_{\tilde{x}}^{\gamma,\delta,h}=\{\Phi_{\tilde{x}}(graph\, \phi)| \phi\in C^{1}(B^{n-k}_{h},\,B^{k}_{h}),\Vert\phi(0)\Vert\leq\delta,\,\Vert d\,\phi\Vert\leq\gamma\},
\end{equation}
where by $B^k_h, B^{n-k}_h$ we denote the ball of radius $h$ around origin respectively in the stable and unstable bundle. A simple transversality argument implies:
%Clearly, for small constants $\gamma, \delta, h$ the above family of submanifolds 
%Let $d(.,.)$ be the distance function generated by the given Riemannian metric on $M$ ,$\tilde{d}(.,.)$ the corresponding distance function on $M^{f}$ as defined in \ref{eq:inverse limit distance} and $d^{\prime}_{\x}(.,.)$ the Lyapunov distance generated by metric $<.,.>_{\x}^{\prime}$.

\begin{proposition}\label{2-2}
Let $\tilde{x}\in \tilde{\Delta}_{l}$ and $h > 0$ small. Then for small constants $\gamma, \delta$ any stable-admissible manifold in $S_{\tilde{x}}^{\gamma,\delta,h}$ intersects any unstable-admissible manifold in $U_{\tilde{x}}^{\gamma,\delta,h}$ at exactly one point and the intersection is transversal. The constants can be chosen universal in a Pesin block.
\end{proposition}

\subsection{Katok Closing Lemma for Endomorphisms}
Let $\mu$ be an ergodic invariant measure for an endomorphism $f: M \rightarrow M$ with non-zero Lyapunov exponents. The following lemma is an endomorphism version of Katok closing lemma for diffeomorphisms. We emphasize that in our context we deal with local diffeomorphisms and we use Lyapunov charts which depend on trajectories. So, a messy adaptation of the same proof of Katok implies the following closing lemma.
\begin{lemma} (Katok Closing Lemma)\label{2-7} 
	Let $f$ be a $C^2-$endomorphism of a compact Riemannian surface $M$. For any positive numbers $l,\delta$ there exists a number $\varrho=\varrho(l,\delta)>0$ such that if for some point $\tilde{x}\in \tilde{\Delta}_{l}$ (Pesin block) and some integer $m$ one has 
	\begin{equation}\label{2-07}
		\tilde{f}^{m}(\tilde{x})\in \tilde{\Delta}_{l}\,\,\,\,\,\,\,and\,\,\,\,\,\,\,\tilde{d}(\tilde{x},\tilde{f}^{m}(\tilde{x}))<\varrho,
	\end{equation} 
	then there exists a point $z\in M$ and $\bar z\in M^{f}$ such that $z=\pi(\bar z)$ and
	\begin{itemize}
       \item $f^{m}( z)= z$ and $\tilde{f}^{m}(\bar z)=\bar z$;
	   \item $d_{n}^{f}(x,z)<\delta$;($\,\,\,d_{n}^{f}$ is defined as $d_{n}^{f}(x,z)= \max_{0\leq i\leq n-1}\,d(f^{i}x,f^{i}z).$ 
	   \item the point $z$ is a hyperbolic periodic point for $f$ and its $W^{s}_{loc}(z)$ and $W^{u}_{loc}(\bar z)$ are admissible manifolds close to $x$, in the chart of $\x$. 
	\end{itemize} 
\end{lemma}

\section{Ergodic Homoclinic Classes and Proof of Theorems} % Main chapter title

%----------------------------------------------------------------------------------------
The notion of ergodic homoclinic classes comes from the work of \cite{13}. The authors defined this notion proving the (at most) uniqueness of SRB measures for surface transitive diffeomorphisms. Here we define a similar notion for endomorphisms.

\subsection{Ergodic Homoclinic Classes.}\label{4-0}
Let $p \in M$ be a hyperbolic periodic point with period $n$. We define the \textit{Ergodic Homoclinic Class} of $p$ both in limit inverse and in the manifold $M.$ Recall that $\bar{p}$ is the unique periodic point of $\tilde{f}$ such that $\pi(\bar{p}) =p.$ That is, $\bar{p}= (\cdots, p, f(p), \cdots f^{n-1}(p), p, \cdots).$ 

The inverse limit Ergodic Homoclinic Class ($\widetilde{EHC}$) is defined as $\tilde{\Lambda}(\bar{p}):=\tilde{\Lambda}^{s}(\bar{p})\cap\tilde{\Lambda}^{u}(\bar{p})$ where,

\begin{equation}
	\tilde{\Lambda}^s(\bar{p}):=\{\x \in \tilde{\mathcal{R}}\big{|} \exists n \geq 0 ,   W_{loc}^{s}(\tilde{f}^n(\x)) \pitchfork W^{u}(\mathcal{O}(\bar{p}))\neq \emptyset \}
\end{equation}
and 
\begin{equation}
	\tilde{\Lambda}^u(\bar{p}):=\{ \x \in\tilde{\mathcal{R}}\big{|} \exists \tilde{y} \in \tilde{\mathcal{R}}, \pi(\tilde{y}) = \pi(\x), \exists n \geq 0, \, \,  f^n(W^{u}(\y)) \pitchfork W_{loc}^{s}(\mathcal{O}(\bar{p}))\neq \emptyset\}
\end{equation}
Here $\tilde{\mathcal{R}}$ denotes regular points in $M^f.$ Observe that $\pi^{-1}(\pi (\tilde{\Lambda}^{*}(\bar{p}))) = \tilde{\Lambda}^{*}(\bar{p}) $ for $* \in \{s, u\}.$ We denote by $\Lambda^s(p) := \pi(\tilde{\Lambda}^s(\bar{p}))$, $\Lambda^u(p) := \pi(\tilde{\Lambda}^u(\bar{p})) $ and $ \Lambda(p) := \pi(\tilde{\Lambda}(\bar{p}))$.

Suppose that $\mu$ is a hyperbolic $f-$invariant Borel probability measure. 
Take $\x \in \tilde{\Delta}_l$ a recurrent point in the support of $\tilde{\mu}$ restricted to the Pesin block $\tilde{\Delta}_l$. Using closing lemma \ref{2-7} we find a hyperbolic periodic point $\bar{p}$ and we prove two following crucial lemmas about the ergodic homoclinic class of $\bar{p}.$

\begin{lemma} \label{4-5}
	Let $p$ be a periodic point obtained as above, then $\tilde{\mu}(\tilde{\Lambda}(\bar p))>0.$
	
\end{lemma}
Let $\tilde{B}$ be a small ball around $\x$ such that $\tilde{\mu}(\tilde{B} \cap \tilde{\Delta}_l) > 0$ . 
By the item (c) in the closing lemma \ref{2-7}, $W^u_{loc}(\bar{p})$ and $W^s_{loc}(\bar{p})$ are respectively close to $W^u_{loc}(\tilde{x})$ and $W^s_{loc}(\x).$
By continuity of stable and unstable manifolds in the Pesin blocks, for any point $\tilde{y}\in \tilde{B} \cap \tilde{\Delta}_l$ we have that $W^u_{loc}(\y)$ and $W^s_{loc}(\y)$ are respectively close to $W^u_{loc}(\tilde{x})$ and $W^s_{loc}(\x)$ in the $C^1-$topology.
Consequently by transversality arguments we conclude that $\tilde{y} \in \tilde{\Lambda}(\bar p).$ 

\begin{lemma}\label{4-6}
	$\tilde{\Lambda}(\bar p)$ is $\tilde{f}-$invariant.
\end{lemma}

 In fact we prove that both $\tilde{\Lambda}^s(\bar{p})$ and $\tilde{\Lambda}^u(\bar{p})$ are invariant. Firstly let us prove that $\tilde{f}(\tilde{\Lambda}(\bar p))\subset \tilde{\Lambda}(\bar p)$. 
Recall that $\tilde{\Lambda}(\bar p)= \tilde{\Lambda}^{s}(\bar p)\cap\tilde{\Lambda}^{u}(\bar p).$
Without loss of generality let suppose that $p$ is a hyperbolic fixed point, then:
\begin{align*}
	\tilde{x}\in \tilde{\Lambda}^{u}(\bar p)&\Rightarrow f^n (W^{u}(\tilde{y}))\pitchfork W_{loc}^{s}(\bar p)\neq\emptyset, \pi(\tilde{y}) = \pi(\x)\\
	&\Rightarrow f^{n+1}(W^{u}(\tilde{y}))\pitchfork W_{loc}^{s}(\bar p)\neq\emptyset\\
	&\Rightarrow f^n(W^{u}(\tilde{f}(\tilde{y}))\pitchfork W_{loc}^{s}(\bar p)\neq\emptyset\\
	&\Rightarrow \tilde{f}(\tilde{x})\in \tilde{\Lambda}^{u}(\bar p).
\end{align*}

On the other hand, if
$\tilde{x}\in \tilde{\Lambda}^{s}(\bar p) \Rightarrow W_{loc}^{s}(\tilde{f}^n(\x))\pitchfork W^{u}(\bar p)\neq\emptyset$ for some $n \geq 0.$ This implies:
\begin{center}
		$\begin{cases}
		 W_{loc}^{s}(\tilde{f}^{n-1} (\tilde{f}(\x))))\pitchfork W^{u}(\bar p)\neq\emptyset, \quad \text{if} \quad n > 0;\\		
		W^s_{loc} (\tilde{f}(\x)) \pitchfork W^{u}(\bar p)\neq\emptyset \quad \text{if} \quad n = 0
		\end{cases}$
	\end{center} 
So, $\tilde{f}(\tilde{x})\in \tilde{\Lambda}^{s}(\bar p).$

%So we have that $\tilde{f}(\tilde{\Lambda}(\bar p))\subset \tilde{\Lambda}(\bar p)$.
Now we prove that $\tilde{\Lambda}(\bar p)\subset\tilde{f}(\tilde{\Lambda}(\bar p))$. We divide the proof into two steps again:
\begin{itemize}
	\item
	$\tilde{\Lambda}^{u}(\bar{p})\subset \tilde{f}(\tilde{\Lambda}^{u}(\bar{p}))$: Take $\tilde{x}=\tilde{f}(\tilde{f}^{-1}(\tilde{x})) \in \tilde{\Lambda}^{u}(\bar{p})$. So by definition there exist $\tilde{y}, \pi(\tilde{y})= \pi(\x)$ and $n \geq 0$ such that $ W^{u}(\tilde{f}^n(\y))) \pitchfork W_{loc}^{s}(\bar{p})\neq \emptyset$ which implies that (by applying $f$)
	
$$W^{u}(\tilde{f}^{n+1}(\tilde{f}^{-1}(\y)))) \pitchfork W_{loc}^s(\bar{p})\neq \emptyset \Rightarrow \tilde{f}^{-1}(\x) \in \tilde{\Lambda}^{u}(\bar{p}).$$

	\item  $\tilde{\Lambda}^{s}(\bar p)\subset \tilde{f}(\tilde{\Lambda}^{s}(\bar{p}))$: Once again taking $\tilde{x}=\tilde{f}(\tilde{f}^{-1}(\tilde{x}))$; by definition of $\tilde{\Lambda}^{s}(\bar p)$ we have:
	
	$$W^{s}(\tilde{f}^{n+1}(\tilde{f}^{-1}(\x))) \pitchfork W^{u}(\bar{p}) = \emptyset$$
		
	which implies $\tilde{f}^{-1}(\x) \in \tilde{\Lambda}^s(\bar{p})$.

\end{itemize}
%Finally observe that  $\pi(\tilde{\Lambda}(\bar p)) = \pi(\tilde{f}(\tilde{\Lambda}(\bar p))= f(\pi(\tilde{\Lambda}(\bar p)))\Rightarrow f(\Lambda(p))=\Lambda(p).$

%Using $\lambda-$lemma we have the following result.

%\begin{lemma}\label{4-8}
%	If $p$ and $q$ are two homoclinically related periodic points, then $\tilde{\Lambda}^{u}(\bar p)=\tilde{\Lambda}^{u}(\bar q)$.
%\end{lemma}

%\begin{proof}
%	To show that $\tilde\Lambda^{u}(\bar p)\subset\tilde\Lambda^{u}(\bar q)$ let take some $\x\in \tilde{\Lambda}^{u}(\bar p)$, we show that $\x\in \tilde{\Lambda}^{u}(\bar q).$  For simplicity let also suppose that, $p$ and $q$ are hyperbolic fixed points. By definition $W^{u}(\tilde{x})\pitchfork W^{s}(p)\neq\emptyset$. Using $\lambda-$lemma for a small disc around $W^{u}(\tilde{x})\pitchfork W^{s}(p)$ there exists $k$ such that $f^{k}(W^{u}(\tilde{x}))$ is sufficiently $C^{1}-$close to $W^{u}(\bar p).$ By homoclinic relation between $p$ and $q$,  
%	\begin{equation}\label{4-9-1}
%		f^{k}W^{u}(\tilde{x})\pitchfork W^{s}(q)\neq\emptyset
%	\end{equation} 
%	and by invariance of $\tilde{\Lambda}^u(\bar{q})$ we conclude that $\x \in \tilde{\Lambda}^u(\bar q). $
	
%\end{proof}
%\begin{figure}[hbtp]
	%\caption{?}
%	\centering
%	\includegraphics[scale=0.9]{lemma01.pdf}
%	\caption{Lambda-lemma}
%	\label{fig:lambda lemma}
%\end{figure}
\subsection{Ergodic Criterion.}

%In this section we are going to show that if $\Lambda(p)$ for $p$ "some hyperbolic periodic point" has positive measure, then it coincide with a hyperbolic ergodic component which contains p. Before giving the main theorem of this part we need to remember once again that f is not invertible necessarily and this may do not give enough information through Birkhoff Ergodic Theorem on $M$. For this reason we will write the limit version of BET on natural extension or inverse limit space of $f$.

In this subsection we give the most important part of the proof and from now on $\mu$ is a measure with SRB property.

\begin{theorem}[Birkhoff Ergodic Theorem for Natural Extension]\label{4-10}
	Let $\tilde f: M^f\to M^f$ be the lift homeomorphism on inverse limit space of $f$ and $\tilde \mu$ the unique $\tilde f-$invariant lift measure of a Borel probability $f-$invariant measure $\mu$ on $M$. Let $\tilde{\phi} \in C(M^{f})$ a continuous function on $M^{f}$. 
	For $\tilde \mu$ almost every point $\tilde x \in M^f$  the following two limits exist $$\tilde{\phi}^{\pm}(\tilde x)=\lim_{n \rightarrow \pm\infty} \frac{1}{n}\sum_{j=0}^{n-1} \tilde{\phi}(\tilde{f}^{j}(\tilde x)).$$ 

Both $\tilde{\phi}^{+}$ and $\tilde{\phi}^{-}$  are  $\tilde \mu-$integrable  $\tilde{f}-$invariant function with $\int \tilde{\phi}^{\pm} d \tilde{\mu} = \int \tilde{\phi} d\tilde{\mu}.$
In particular if $\tilde{\mu}$ is ergodic then $\tilde{\phi}^{\pm}$ are constant functions. \end{theorem}
 
 \begin{remark}
 It comes out from the proof of Birkhoff ergodic theorem that for $\tilde \mu-$a.e $\x\in M^f$,
  $\tilde{\phi}^{+}(\tilde x)=\tilde{\phi}^{-}(\tilde x)$. Such points are called \textbf{typical points}.
\end{remark}

\begin{lemma}\label{4-11}
	There exists an invariant set $\tilde{S}$ of typical points with $\tilde{\mu}(\tilde{S})=1$ such that for all $\tilde{\phi}\in C(M^f)$, if $\tilde x\in\tilde{S}$ then for all $\tilde{\omega}\in \widetilde{W^{s}}(\tilde x)$ and $\tilde{\mu}^{u}_{\tilde x}-$a.e $\tilde{\zeta}\in \widetilde{W^{u}}(\tilde x)$, $$\tilde{\phi}^{+}(\tilde x)= \tilde{\phi}^{+}(\tilde{\zeta}) = \tilde{\phi}^{+}(\tilde{\omega}).$$
\end{lemma}
\begin{proof}
We claim that for almost all typical points $\x$, we have $\tilde{\mu}^{u}_{\tilde x}-$a.e $\tilde{\zeta}\in\widetilde{W^{u}_{loc}}(\tilde x)$ is typical. The proof of this claim is mutatis mutandis to the lemma 3.1 in \cite{13}. Indeed, if there exists a subset of $\tilde{\mu}-$positive measure which does not satisfy the above claim, then using the definition of conditional measures we get a contradiction to the fact that typical points has full $\tilde{\mu}-$measure.

We take $\tilde{S}$ as the full $\tilde{\mu}-$measure subset of points $\x$ obtained above.
	
	Observe that by definition of typical points: \begin{align}\label{4-11-1}
	\bar{\phi}^{+}(\tilde{x})=\bar{\phi}^{-}(\tilde{x})\,\,\,\,\, and \,\,\,\,\, \bar{\phi}^{+}(\tilde{\zeta})=\bar{\phi}^{-}(\tilde{\zeta}),
	\end{align}
	and from continuity of $\tilde{\phi}$:
	\begin{align}\label{4-11-2}
	\tilde{\phi}^{-}(\tilde{\zeta})=\tilde{\phi}^{-}(\tilde{x}) \, \text{for all} \,\,\,\,\, \tilde{\zeta}\in\widetilde{W^{u}}(\tilde x).
	\end{align}
	
	Using \label{4-11-2}, \ref{4-11-1} we conclude that $\tilde{\phi}^{+}(\tilde x)=\tilde{\phi}^{+}(\tilde{\zeta})$.
	
	Again by continuity of $\tilde{\phi}$ and using the definition of stable sets,  we conclude that
	$\tilde{\phi}^{+}(\tilde x) = \tilde{\phi}^{+}(\tilde{\omega})$ for every $\tilde{\omega} \in \widetilde{W^{s}}(\tilde x).$
	
	By the definition of ergodic sum it is clear that $\tilde{S}$ is an invariant set.

\end{proof}
%\textbf{\textit{Observation.}} For any $\tilde A\subset M^f$ with $\tilde{\mu}(\tilde A)>0$ we have $\mu(\pi(\tilde A))=\tilde{\mu}(\pi^{-1}(\pi(A)))\geq \tilde{\mu}(\tilde A)>0$. Also for each $A\subset M$ with $\mu(A)>0$ we have $\tilde{\mu}(\pi^{-1}(A))=\pi_{*}\tilde{\mu}(A)=\mu(A)>0$.
\begin{lemma}\label{4-12}
	Given $\tilde{\phi}\in \mathcal{L}^{1}(M^{f})$, there exists an invariant set $\tilde{S}_{\tilde{\phi}}\subset M^{f}$ with $\tilde{\mu}(\tilde{S}_{\tilde{\phi}})=1$ such that if $\tilde{x}\in \tilde{S}_{\tilde{\phi}}$ then $\tilde{\mu}^{u}_{\tilde{x}}-$a.e $\tilde{y}\in\widetilde{W}^{u}(\tilde{x})$ satisfies $$\tilde{\phi}^{+}(\tilde{y})=\tilde{\phi}^{+}(\tilde{x}).$$
\end{lemma}
\begin{proof}
	Given $\tilde{\phi}\in\mathcal{L}^{1}(M^{f})$, as
	%a consequence of Urysohn Lemma we can see that on the metric space $M^{f}$, (what is this space$\mathcal{L}^{1}(M^{f})$? which norm it has???)
	$C(M^{f})$ is dense in $\mathcal{L}^{1}(M^{f})$ therefore we can take a sequence of continuous functions $\tilde{\phi}_{n}$ converging to $\tilde{\phi}$ in $L^1-$topology. 
 By Birkhoff Ergodic Theorem for natural extensions, $\tilde{\mu}-$a.e $\tilde{x}$, $\tilde{\phi}_{n}^{+}(\tilde{x})$ exists and $\tilde{\phi}_{n}^{+}$ converges to $\tilde{\phi}^{+}$ in $L^1-$topology.

As $M^{f}$ is a compact metric space, there exists a sub-sequence $\tilde{\phi}^{+}_{n_{k}}$ and a full $\tilde{\mu}$ measure subset $\tilde{J}$ such that for every $\tilde{x} \in \tilde{J}$ we have $\tilde{\phi}^{+}_{n_{k}} (\x)$ converge to $\tilde{\phi}^{+}(\tilde{x})$.

Now take $\tilde{S}_{\tilde{\phi}}:= \tilde{J} \cap \tilde{S}$ and the proof is complete.

\end{proof}

\begin{theorem}\label{4-13}
	Let $f:M\rightarrow M$ be a $C^2-$endomorphism over a compact manifold $M$ equipped with a hyperbolic measure $\mu$ with SRB property. If $\tilde{\mu}(\tilde{\Lambda}(\bar p))>0$ then
	$$\tilde{\Lambda}^{u}(\bar p)\subset^{\circ}\tilde{\Lambda}^{s}(\bar p).$$
\end{theorem}

\begin{proof}
	First remember that by definition $\tilde{\Lambda}(\bar p)=\tilde{\Lambda}^{u}(\bar p)\cap \tilde{\Lambda}^{s}(\bar p).$ 
		By lemma \ref{4-12} it is enough to prove that $ \tilde{\Lambda}^{u}(\bar p)\cap\tilde{S}_{1_{\tilde{\Lambda}^{s}}} \subset \tilde{\Lambda}^{s}(\bar p)$.\\ %then $x\in \Lambda^{u}(p)\cap S_{1}$ implies $x\in\Lambda^{s}(p)$.
	The following two lemmas are useful in the proof.	
		\begin{lemma}\label{4-13-1}
		If exists a $\tilde{\mu}^{u}_{\tilde{x}}-$positive measure subset of
		$\widetilde{W}^{u}(\tilde{x})$ belonging to $\tilde{\Lambda}^{s}(\bar p)$, then $\tilde{x}\in \tilde{\Lambda}^{s}(\bar p)$.
	\end{lemma}
	\begin{proof}
		The $\tilde{\Lambda}^{s}(\bar p)$ is $\tilde{f}-$invariant and $1_{\tilde{\Lambda}^{s}(\bar p)}=\tilde{f}(1_{\tilde{\Lambda}^{s}(\bar p)})$. This implies that
		if\,$\tilde{x}\notin \tilde{\Lambda}^{s}(\bar p)$ then $\tilde{\mu}^{u}_{\tilde{x}}-a.e. \tilde{y}\in \widetilde{W}^{u}_{loc}(\tilde{x})$ does not belong to $\tilde{\Lambda}^{s}(\bar p)$.
	\end{proof}
	\begin{lemma}\label{4-13-2}
		If exists some $m\in \mathbb{N}$ such that a $\tilde{\mu}^{u}_{\tilde{f}^{m}(\tilde{x})}-$positive measure subset, of
		$\widetilde{W}^{u}(\tilde{f}^{m}(\tilde{x}))$ belong to $\tilde{\Lambda}^{s}(\bar p)$, then $\tilde{x}\in \tilde{\Lambda}^{s}(\bar p)$.
	\end{lemma}
	\begin{proof}
		This comes from the fact that $\tilde{x}\in\tilde{S}_{1_{\tilde{\Lambda}^{s}}}$ and $\tilde{S}_{1_{\tilde{\Lambda}^{s}}}$ is an $\tilde{f}-$invariant set. The rest will be a corollary of last lemma.
	\end{proof}

 Take $\tilde{y}\in\tilde{\Lambda}^{s}(\bar p)$ an auxiliary point in a way that for some $ l>0$ both $\tilde{x},\tilde{y}$ lies in the same Pesin block $\tilde{\Delta}_{l}$ and $\tilde{y}\in supp(\tilde{\mu}|_{\tilde{\Delta}_{l}\cap\tilde{\Lambda}^{s}(\bar p)})$. We additionally suppose that $\tilde{y}$ returns back to $\tilde{\Delta}_{l}\cap\tilde{\Lambda}^{s}(\bar p)$ infinitely many times.
 
	As $\tilde{y}\in\tilde{\Lambda}^{s}(\bar p)$ then by definition there exists $n \geq 0$ such that $W_{loc}^{s}(\tilde{f}^n(\y))\pitchfork W^{u}(\bar p)\neq\emptyset$. Without loss of generality we suppose that $n=0.$ As $ W^{s}_{loc}(\y)\pitchfork W^{u}(\bar p)\neq\emptyset$  for large enough $n$ we have that $f^{n}(y)$ is very close to $W^{u}(\bar p).$
		Using Poincar\'e recurrence theorem, we could choose $n$ in such a way that $\tilde{f}^{n}(\tilde{y})\in \tilde{\Delta}_{l}$ and put $\tilde{\alpha}:= \tilde{f}^{n}(\y).$
	
	By definition $\x \in \tilde{\Lambda}^{u}(\bar p)$. So again without loss of generality we suppose $W^{u}(\tilde x)\pitchfork W^{s}_{loc}(\bar p)\neq\emptyset$ and $\tilde{\alpha}$ is such that $W^{s}_{\delta}(\tilde{\alpha})\pitchfork W^{u}(\bar p)\neq\emptyset$. Using $\lambda-$lemma we find some large $m$ that $\tilde{f}^{m}(\x)\in\tilde{\Delta}_{l}$ and
	\[W^{u}(\tilde{f}^{m}(\tilde{x}))\pitchfork W^{s}_{\delta}(\tilde{\alpha})\neq \emptyset.\,\,\,\,\,\,\,\,(*)\]
	By hypothesis $\mu$ is hyperbolic with SRB property and $\pi_{*}(\tilde{\mu}^{u}_{\tilde{x}})\approx m^{u}_{\tilde{x}}$. Let call $\pi_{*}(\tilde{\mu}^{u}_{\tilde{x}})=\mu^{u}_{\tilde{x}}$. We are going to find a positive $\tilde{\mu}^{u}_{\tilde{x}}-$subset of $\widetilde{W}^{u}(\tilde{x})$ belonging to $\tilde{\Lambda}^{s}(\bar p)$. Using lemma \ref{4-13-2} it is enough to find it on local unstable manifold of some iterate of $\x$.
	\begin{figure}[hbtp]
		%\caption{?}
		\centering
		\includegraphics[scale=0.95]{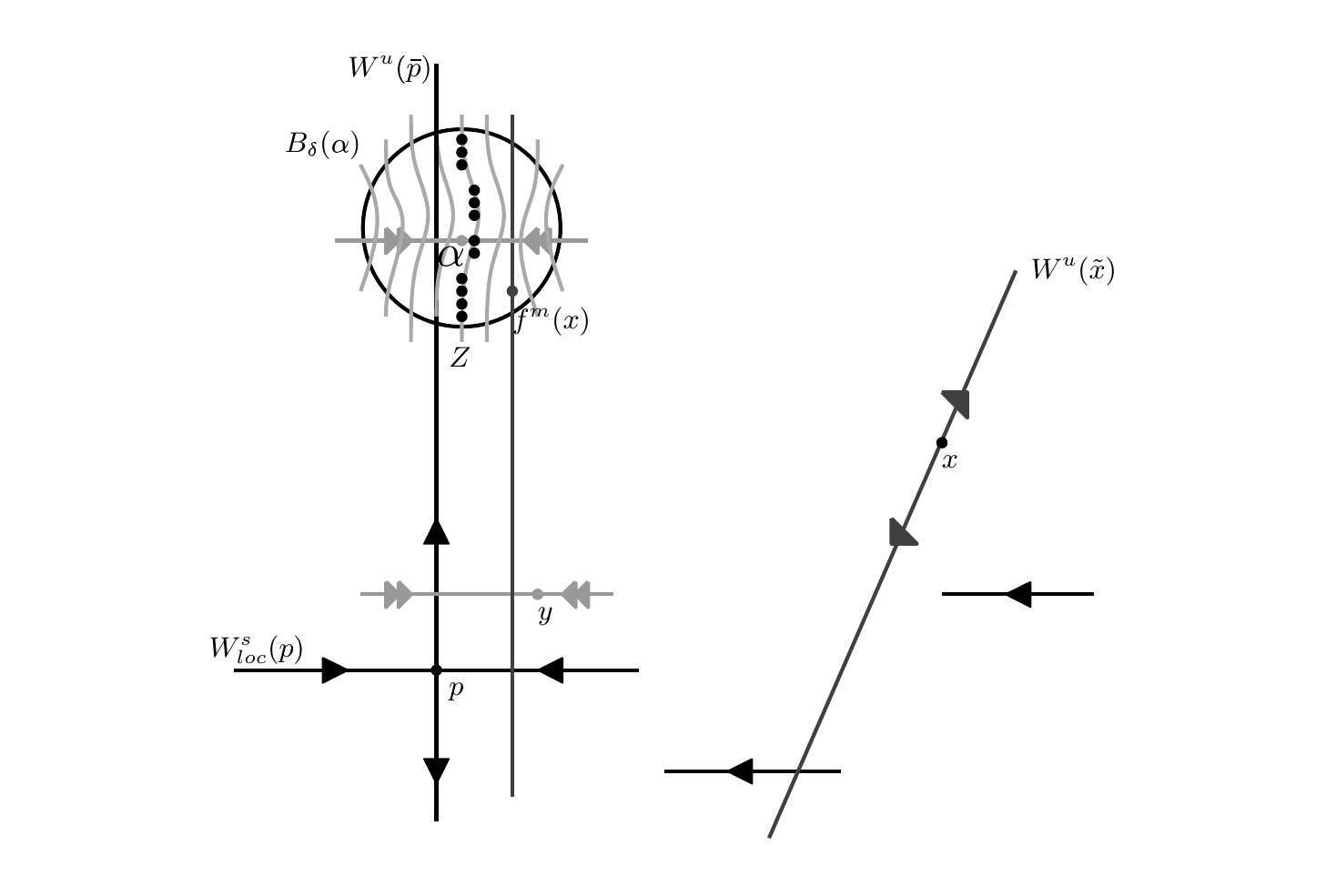}
		\caption{Ergodic Criterion}
		\label{fig:ergodic criteria}
	\end{figure}
	
	For this purpose, 
	%Let choose a density point of $\tilde{\Delta}_{l}$.
	%The $\tilde{\alpha}:=\tilde{f}^{n}(\tilde{y})$ which projects to $\alpha$, belongs to $\tilde{\Delta}_{l}$.
	consider a very small ball $\tilde{B}_{\delta}(\tilde{\alpha})$ around $\tilde{\alpha}$. Covering $\tilde{B}_{\delta}(\tilde{\alpha})$ with a measurable partition sub-ordinate to unstable manifolds, there exists some point $\tilde{z}\in \tilde{\Delta}_{l}$ such that $$\tilde{\mu}^{u}_{\tilde{z}}(\tilde{S}|_{\tilde{\Lambda}^{s}(\bar p)}\cap\tilde{\Delta}_{l}\cap \tilde{B}_{\delta}(\tilde{\alpha}))>0.$$
	Let us call this subset of $\widetilde{W}_{loc}^{u}(\tilde{z})$ by $\tilde{Z}:= \tilde{S}|_{\tilde{\Lambda}^{s}(\bar p)}\cap\tilde{\Delta}_{l}\cap \tilde{B}_{\delta}(\tilde{\alpha}))$. The  projection $Z:= \pi(\tilde{Z})$ is inside $W^{u}(\tilde{z})\cap \pi( B(\tilde{\alpha}))$. 
	
	By definition of metric in the orbit space $M^f$, it is clear that $\pi (B(\tilde{\alpha}))$ is a ball of smaller radius (than the radius of $B(\tilde{\alpha})$) around $\alpha:= \pi(\tilde{\alpha}).$ So, $z :=\pi(\tilde{z})$ is close enough to $\alpha.$
	
	 Observe that stable lamination varies continuously in a Pesin block and as a consequence of transversality of $W^s_{\delta}(\tilde{\alpha})$ and $W^u(\bar{p})$ and $C^1-$closeness of $W^u(\tilde{f}^m(\x))$ to $W^u(\bar{p})$ one can define the stable holonomy map from $W^u(\tilde{z})$ into both $W^u(\tilde{f}^m(\x))$ and $W^u(\bar{p}).$  The domain of the holonomy map at least contains $Z.$

	By SRB property we know that $m_{\tilde{z}}^u (Z) > 0$ and using absolute continuity of stable holonomy into $W^u(\tilde{f}^m(\x))$, we have $m^{u}_{\tilde{f}^{m}(\tilde{x})}(Z^*) > 0$ where $Z^*$ is the image of $Z$ by stable holonomy. Again using SRB property (equivalence of conditional measures and Lebesgue measure) it comes out that $\mu^{u}_{\tilde{f}^{m}(\tilde{x})}(Z^{*})>0.$ Now using definition of $\tilde \mu^{u}$ we have:
	 \[\tilde\mu^{u}_{\tilde{f}^{m}(\tilde{x})}(\pi^{-1}(Z^{*}))=\mu^{u}_{\tilde{f}^{m}(\tilde{x})}(Z^{*})>0.\]
	 Finally observe that any point in $\pi^{-1}(Z^*)$ belongs to $\tilde{\Lambda}^s (\bar{p})$ as its stable manifold intersects $W^u(\bar{p})$ and this completes the proof.
	
\end{proof}

%\begin{corollary}
%	An immediate conclusion of the above theorem will be that \[\Lambda^{u}(p)\subset^{\circ}\Lambda^{s}(p).\]
%\end{corollary}
\begin{theorem}\label{4-14}
	Let $f:M\rightarrow M$ be a $C^2-$endomorphism over a compact manifold $M$ and $\mu$ any measure with SRB property. If $\tilde{\mu}(\tilde{\Lambda}(\bar p))>0$ for a hyperbolic periodic point $\bar p$, then $\tilde{\mu}|_{\tilde{\Lambda}(\bar p)}$ is an ergodic component of $\tilde{\mu}$.
\end{theorem}
\begin{proof}
	
	For simplicity once again let assume that $\bar p$ is a hyperbolic fixed point. As $\tilde{\mu}(\tilde{\Lambda}(\bar p))>0$, taking any $\tilde{f}-$invariant continuous function $\tilde{\phi}: M^{f}\rightarrow \mathbb R$, we show that $\tilde{\mu}$-a.e. points in $\tilde{\Lambda}(\bar p)\cap \tilde{J}$ ($\tilde{J}$ is the set of typical points from \ref{4-11}) are $\bar{\phi}^{+}-$constant. Which implies the ergodicity of $\tilde{\mu}|_{\tilde{\Lambda}(\bar p)}$.
	Let choose arbitrary $\tilde{x},\tilde{y}\in \tilde{\Delta}_{l}\cap\tilde{\Lambda}(\bar p)\cap \tilde{S_1}:=\tilde{\Gamma}$ for some $l>0$. Without loose of generality we may assume that such $\tilde{x},\tilde{y}$ are in the support of $\tilde{\mu}|_{\tilde{\Lambda}(\bar p)}$. Using Poincar\'{e} recurrence theorem these points come back infinitely many times to $\tilde{\Gamma}$.
	Following a similar argument to theorem \ref{4-13} we prove that $\tilde{\phi}^+(\x) = \tilde{\phi}^ +({\tilde{y}}).$ As the argument is exactly the same, just substituting $1_{\tilde{\Lambda}^s(\bar p)}$ to $\tilde{\phi}$ we do not repeat it here.
	
\begin{corollary} \label{keypoint}
Let $f: M \rightarrow M$ be a $C^2-$endomorphism and $\mu$ any ergodic SRB measure of index $0 < k < n.$ Then there exists a hyperbolic periodic point $p_{\mu}$ in  $Per_{k}$ such that $\tilde{\mu}(\tilde{\Lambda}(\bar{p}_{\mu})) =1.$ 
\end{corollary}	
\begin{proof}
Observe that by lemma \ref{4-5} we have a periodic point with ergodic homoclinic class of positive measure. The above theorem and the ergodicity of $\mu$ imply that $\tilde{\mu}(\tilde{\Lambda} (\bar{p}_{\mu}))=1.$
\end{proof}
\end{proof}

%\begin{remark}
%	In above theorem $\hat{\phi}$ is supposed to be Birkhoff forward sum for the $f-$invariant $\phi(x)=\tilde{\phi}[\tilde{x}]$. Where $\tilde{x}\in \pi^{-1}(x)$ and $[x]$ contains all the pre-images projecting on $x$. The $\phi$ in this way is well defined and it is possible to show that by $\tilde{f}-$ invariance of $\tilde{\phi}$ then
%	\[\forall \tilde{y}\in [\tilde{x}]\Rightarrow\tilde{\phi}(\tilde{y})=\tilde{\phi}(\tilde{x}).\]
%\end{remark}
\section{Proof of the Main Theorem} 

 We have proved that ergodic homoclinic classes are in a close relationship with ergodic components of a measure. In fact for any two hyperbolic ergodic measure $\mu$ and $\nu$ with SRB property, we show that they are supported on the ergodic homoclinic class of some periodic point respectively  $p_{\mu}$ and $p_{\nu}$. Then we prove that if either $[p_{\mu}, p_{\nu}] \neq \emptyset$ or $[p_{\nu}, p_{\mu}] \neq \emptyset$, then the measures are the same.

Although the base of this work is settled on the assumption of ergodic hyperbolic measures, using ergodic decomposition theorem \cite{23}, theorem \ref{5-2} and proposition \ref{5-3} we can reduce the proof in the ergodic case. Theorem II.1.1 of \cite{10} gives a version of Margulis-Ruelle inequality for $C^{1}-$maps.

\begin{theorem}
	Lef $f$ be a $C^{1}-$map of a compact, smooth Riemannian manifold $M$. If $\mu$ is an $f-$invariant Borel probability measure on $M$, then
	$$h_{\mu}(f)\leq \int_{M}\,\sum_{i}\,\lambda_{i}^{+}(x)\,m_{i}(x)\,d\mu(x),$$
	where $-\infty<\lambda_{1}(x)<\cdots<\lambda_{r(x)}<\infty$ are Lyapunov exponents of $f$ at $x$ and $m_{i}(x)$ is the multiplicity of $\lambda_{i}(x)$ for each $i=1,2,..,r(x).$
\end{theorem}

\begin{proposition}\label{5-3}
	Almost all ergodic components of $\mu$ are hyperbolic and SRB.
\end{proposition}

\begin{proof}
	The hyperbolicity is easy to see because if not it would be possible to find a positive measure set of points with zero Lyapunov exponent and this contradicts the fact that $\mu$ is hyperbolic. For SRB property, we know that by ergodic decomposition and Margulis-Ruelle inequality, there exists a probability measure $\hat \mu$ in the space of all probability measures supported on ergodic measures $\mathcal{M}(f)$, such that  $h_{\mu}=\int_{\mathcal{M}(f)}h_{\nu}\,d\hat{\mu}(\nu)\leq \int\sum_{i} \lambda^{+}_{i}(x)\,d\mu$,  counting multiplicities. By theorem VII.1.1 of  \cite{10} $\mu$ has SRB property if and only if
	\[h_{\mu}(f)=\int_{M}\sum_{i}\lambda_{i}^{+}(x)\,d\mu(x).\]
	These clearly imply that $\hat{\mu}-$almost every $\nu$ will satisfy the entropy formula and so has SRB property.
\end{proof}

\begin{theorem}\label{5-2}
$f:M\rightarrow M$ a $C^2-$endomorphism over a compact manifold $M$ equipped with a hyperblic measure $\mu$ with SRB property and $\tilde{\mu}$ its lift. Then for any ergodic component $\tilde{\nu}$ of it, there exists a hyperbolic periodic point $p$ such that $\tilde{\nu}(\tilde{\Lambda}(\bar{p}))=1$.
\end{theorem}

\begin{proof}
Suppose that $\nu$ is an ergodic component of a hyperbolic measure $\mu$ with SRB property. By proposition \ref{5-3} the ergodic components are also hyperbolic with SRB property. By proposition \ref{1} we know the existence of the unique ergodic $\tilde\nu$ such that $\pi_{*} (\tilde\nu) = \nu$. By corollary \ref{keypoint} we get the desired periodic point.

%and consequently $\nu(\Lambda(p))=1$.
%More precisely for $l>0$ take $\tilde{\Delta}_{l}$ with positive $\tilde{\mu}$- measure that the size of stable and unstable local manifolds in M, are uniformly bounded from below by a constant larger than zero. Let $\tilde{x}\in Supp(\tilde{\mu}|_{\tilde{\Delta}_{l}})$ and $\tilde{B}$ be a small ball around $\tilde{x}$ such that $\tilde{\mu}(\tilde{B}\cap\tilde{\Delta}_{l})>0$. By closing lemma, we can find a periodic point $\bar p$ near enough to $\tilde{x}$ such that related stable and unstable local manifolds on $M$, are respectively $C^{1}$ close to $W^{s}_{loc}(x)$ and $W^{u}_{loc}(\tilde{x})$. Consequently this implies $x\in \Lambda(p)$. The $\tilde{\Delta}_{l}$ is a Pesin block so we choose it in a way that stable- unstable laminations vary continuously on it. Then for every $y\in B\cap \Delta_{l}$ where $B\cap \Delta_{\l}$ is the projection of $\tilde{B}\cap\tilde{\Delta}_{l}$ on M, we obtain that $y\in \Lambda(p)$. This yields that $\mu(\Lambda(p))>0$ and the ergodicity of $\mu$ implies that $\mu(\Lambda(p))=1$.
\end{proof}

Let $\mu$ and $\nu$ be ergodic hyperbolic measures with SRB properties with respective periodic points $p_{\mu}$ and $p_{\nu}.$ Assume that $[p_{\mu}, p_{\nu}] \neq \emptyset.$

Let $B(\tilde \mu)$ and $B(\tilde \nu)$ be respectively the basins of $\tilde \mu$ and $\tilde \nu:$
\[B(\tilde\mu)=\{\tilde x:\lim_{n\rightarrow \infty}\frac{1}{n}\sum^{n-1}_{0}\tilde{\phi}(\tilde{f}^{i}(\tilde x))= \int\tilde \phi d\tilde \mu\,\,\,\,\,\,\,\; \forall\,\tilde\phi\in C(M^{f})\};\]

\[B(\tilde\nu)=\{\tilde x: \lim_{n\rightarrow \infty}\frac{1}{n}\sum^{n-1}_{0}\tilde{\phi}(\tilde{f}^{i}(\tilde x))= \int\tilde \phi d\tilde \nu\,\,\,\,\,\,\,\; \forall\,\tilde\phi\in C(M^{f})\};\]

By ergodicity $\tilde{\mu}(\tilde{\Lambda}(\bar p_{\mu}))=\tilde{\nu}(\tilde{\Lambda}(\bar p_{\nu}))=1$ and by $\widetilde{BET}$ \ref{4-10}, we can define $B_{\tilde{\mu}}$ and $B_{\tilde{\nu}}$ with $\tilde{\mu}(B_{\tilde{\mu}})=\tilde{\nu}(B_{\tilde{\nu}})=1$ as follows:

\[B_{\tilde{\mu}}=\{\tilde x:lim_{n\rightarrow \pm\infty}\frac{1}{n}\sum^{n-1}_{0}\tilde{\phi}(\tilde{f}^{i}(\tilde{x}))= \int\tilde{\phi} d\mu\,\,\,\,\,\,\,\; \forall\,\tilde{\phi}\in C(M^{f})\};\]

\[B_{\tilde{\nu}}=\{\tilde x:lim_{n\rightarrow \pm\infty}\frac{1}{n}\sum^{n-1}_{0}\tilde{\phi}(\tilde{f}^{i}(\tilde{x}))= \int\tilde{\phi} d\nu\,\,\,\,\,\,\,\; \forall\,\tilde{\phi}\in C(M^{f})\};\]

It means that $\tilde{\mu}$ (resp. $\tilde{\nu}$)-a.e. point $\tilde{x}\in\tilde{\Lambda}(\bar p_{\mu})$ (resp. $\tilde{\Lambda}(\bar{p}_{\nu})$) belongs to $B_{\tilde{\mu}}$ (resp. $B_{\tilde{\nu}}$). If we show that $B_{\tilde{\mu}}\cap B_{\tilde{\nu}}\neq\emptyset$ then we are done. 

Let us take $\tilde{x}\in\tilde{\Lambda}(\bar{p}_{\mu})$ a point for which $\tilde{\mu}^{u}_{\tilde{x}}((B_{\tilde{\mu}}\cap \tilde{\Lambda}(\bar{p}_{\mu}))^ c)=0$. By SRB property of $\mu$  we will have $m^{u}_{\tilde{x}}((B_{\mu}\cap \Lambda(p_{\mu}))^ c)=0$ where $B_{\mu}:=\pi(B_{\tilde{\mu}})$.
%For any partition subordinated to $W^{u}$ manifolds, specially the one made on local unstable manifolds here, we have that $\pi(\tilde{\mu}_{\tilde}^{})$

There exists some large Pesin block $\tilde{\Delta}_{l}, l\geq 1$ and $\tilde{y}\in\tilde{\Delta}_{l} \cap \tilde{\Lambda}(\bar{p}_{\nu})$,  a density point of $\tilde{\nu}$ such that   $m^{u}_{\tilde{y}}(B_{\nu}\cap \Delta_{l})>0$. Here $B_{\nu}:=\pi(B_{\tilde{\nu}}),\Delta_{l}=\pi(\tilde{\Delta}_{l})$. By Poincar\'e recurrence theorem $\tilde{y}$ returns back infinitely many times to $\tilde{\Delta}_{l}\cap B_{\tilde\nu}$ and consequently $y$ to $\Delta_{l}\cap B_{\nu}$. 

As $\tilde{y}\in \tilde{\Lambda}(\bar{p}_{\nu})$, there exists large iterate  $\tilde{\alpha}=\tilde{f}^{n}(\tilde{y})$ such that $W^{u}(\tilde\alpha)$ becomes very close to $W^{u}(\bar{p}_{\nu})$ in a similar way that has been explained in ergodic criteria section. Figure \ref{fig:ergodic criteria}
%($\delta>0$ small). 

We have $\tilde{x}\in \tilde{\Lambda}(\bar{p}_{\mu})$ and $[p_{\mu}, p_{\nu}] \neq \emptyset.$ So, using $\lambda-$lemma we may find some large iterate $\tilde{f}^{m}(\tilde{x})$ such that $W^{u}(\tilde{f}^{m}(\tilde{x}))$ becomes close enough to $W^{u}(\bar{p}_{\nu})$ in a way that for a positive ${\nu}^{u}_{\tilde{\alpha}}-$measure $z \in \Delta_{l}\cap B_{\nu}$ we have $W^{s}_{loc}(z)\pitchfork W^{u}(\tilde{f}^{m}(\tilde{x}))\neq\emptyset$.

The stable lamination on a Pesin block is absolutely continuous \cite{10} and this implies that $m^{u}_{\tilde{f}^{m}(\tilde{x})}(B_{\nu}\cap \Delta_{l})>0$. We also have $m^{u}_{\tilde{f}^{m}(\tilde{x})}((B_\mu)^ c)=0$ which  implies $B_{\mu}\cap B_{\nu}\neq\emptyset$ and finishes the proof.

\section{Kan example} \label{kan}

The Kan example is a local diffeomorphism $F$ defined on the cylinder $\mathbb{S}^1 \times [0,1]$ as a skew product:
$$
F(z, t):= (z^d, f_z(t)), 
$$ where $z \in \mathbb{S}^1$ is a complex number of norm one and $z^d$ is the expanding  covering of the circle of degree $d > 1$. For each $z \in \mathbb{S}^1$ the function $f_z: [0, 1] \rightarrow [0,1]$ is a diffeomorphism fixing the boundary of $[0, 1].$ Take two fixed points of $z^d$ called $p, q.$ We require that $f_p$ and $f_q$ have exactly two fixed ponits each, a source at $t=1$ (respectively $t=0$) and a sink at $t=0$ (respectively $t=1$). Furthermore, $|f_z^{'} (t)| < d$ and 
$$
 \int \log f_z^{'}(0) dz < 0 \quad \text{and} \quad  \int \log f_z^{'}(1) dz < 0.
$$ 
Under these conditions $F$ has two intermingles SRB measures which are normalized Lebesgue measure of each boundary circle. Under some more conditions $F$ is also transitive (see \cite{BDV}.)
We consider two such examples and glue them  to find a local diffeomorphism of $\mathbb{T}^2$ admitting two SRB measures and topologically transitive.
Take $G : \mathbb{S}^1 \times [0,1] \rightarrow \mathbb{S}^1 \times [0,1] $ as follows:
\begin{equation} \label{doublekan}
G(z, t)= \left\{
\begin{array}{ll}
 (z^d, 1 - \frac{1}{2} f_z(2t)) \qquad &  0 \leq t \leq \frac{1}{2}\\
 (z^d, \frac{1}{2}f_z(2(1-t)) )  &   \frac{1}{2} \leq t \leq 1.
\end{array} \right.
\end{equation}

Observe that the two circles $\{t=0\}, \{t=1/2\}$ are invariant and support SRB measures with intermingles basin on the $2-$torus. We can see also that both  $F$ and $G$ are transitive. However, $G^2$ lets invariant each half trous and consequently $G$ is not mixing. We are not aware of topologically mixing example of systems with intermingled basins of SRB measures. 

Using the proof of our main theorem and the fact that the Lebesgue measures on each invariant circle $\{t=0\}$ and $\{t=1/2\}$  are hyperbolic SRB  measures we can conclude that the number of SRB measures of $G$ is precisely two (without much geometric information about the volume of their basins).

Let $\mu_1$ and $\mu_2$ be respectively the normalized Lebesgue measure on the two invariant circles. It is easy to see that each of these circles is the ergodic homoclinic class of fixed points $p_{\mu_1}:= p, p_{\mu_2} := q$ corresponding to each SRB measure ($p, q$ defined above).  Suppose that there exists another hyperbolic ergodic SRB measure $\nu.$ By  corollary \ref{keypoint} there exits a hyperbolic periodic point $P_{\nu}$ such that $\mu(\Lambda(p_{\nu}))= 1.$ Observe that $p_{\nu}$ is a periodic point in the torus minus two invariant circles. Taking large iterates of the local unstable manifold of $p_{\nu}$ we get a large curve transversal to the stable manifold of $p$ and $q$. That is $[p_{\nu}, p_{\mu_1}] \neq \emptyset$ and $[p_{\nu}, p_{\mu_2}] \neq \emptyset.$ By the proof of the main result $\nu=\mu_1 = \mu_2$ which is an absurd.

%\section{\textbf{Examples}}

\end{document}